\documentclass[preprint,12pt]{amsart}

\usepackage{amssymb}
\usepackage{amsmath}
\usepackage{amsthm}
\usepackage{comment}
\usepackage{float}
\usepackage{hyperref}
\usepackage{graphicx}
\usepackage[margin=1in]{geometry}

\numberwithin{equation}{section}
\newcommand{\ddn}{{\mathrm{d}}}
\newcommand{\dd}{\,{\mathrm{d}}}
\newtheorem{theorem}{Theorem}[section]
\newtheorem{lemma}[theorem]{Lemma}

\newtheorem{remark}[theorem]{Remark}
\newtheorem{definition}[theorem]{Definition}

\begin{document}

%\begin{frontmatter}

%% Title, authors and addresses

%% use the tnoteref command within \title for footnotes;
%% use the tnotetext command for theassociated footnote;
%% use the fnref command within \author or \address for footnotes;
%% use the fntext command for theassociated footnote;
%% use the corref command within \author for corresponding author footnotes;
%% use the cortext command for theassociated footnote;
%% use the ead command for the email address,
%% and the form \ead[url] for the home page:
%% \title{Title\tnoteref{label1}}
%% \tnotetext[label1]{}
%% \author{Name\corref{cor1}\fnref{label2}}
%% \ead{email address}
%% \ead[url]{home page}
%% \fntext[label2]{}
%% \cortext[cor1]{}
%% \address{Address\fnref{label3}}
%% \fntext[label3]{}

\title{A Mixed Discrete--Continuous Fragmentation Model}

%% use optional labels to link authors explicitly to addresses:
%% \author[label1,label2]{}
%% \address[label1]{}
%% \address[label2]{}

\author{Graham Baird}

\author{Endre~S{\"u}li}

\address{Mathematical Institute, University of Oxford, Woodstock Road, Oxford OX2 6GG, UK
email: \texttt{graham.baird@maths.ox.ac.uk}, \texttt{endre.suli@maths.ox.ac.uk}}

\begin{abstract}
\noindent Motivated by the occurrence of ``shattering'' mass-loss observed in purely continuous fragmentation
models, this work concerns the development and the mathematical analysis of a new class of hybrid discrete--continuous fragmentation models. Once established, the model, which takes the form of an integro-differential equation coupled with a system of ordinary differential equations, is subjected to a rigorous mathematical analysis, using the theory and methods of operator semigroups and their generators. Most notably, by applying the theory relating to the Kato--Voigt perturbation theorem, honest substochastic semigroups and operator matrices, the
existence of a unique, differentiable solution to the model is established. This solution is also shown to preserve nonnegativity and conserve mass.
\end{abstract}



%\end{frontmatter}

\maketitle 

%% \linenumbers

%% main text
\section{Introduction}

\noindent The mathematical modelling of fragmentation, and the reverse coagulation process have a long history, with the first work dating back to \cite{smol17}. Such models have found applications in areas as diverse as polymer science \cite{ziff80}, population dynamics \cite{degond17} and astrophysics \cite{johansen08}. The models of such processes typically classify the entities within the system according to some physical state variable, for example their volume, area or mass, the aim being then to determine the evolution of the system with respect to this variable as time progresses. Models are typically classified as either discrete or continuous, depending on the nature of the state variable of interest. Generally, there are great similarities between the two forms, with each continuous model having an analogous discrete version and vice versa. The selection of the form to use is largely a modelling choice, and depends on the scale of the phenomenon to be described.\\

\noindent In this paper we shall be exclusively considering fragmentation processes, with no coagulation mechanism involved. With continuous models of such processes, difficulties can arise when the break-up rate for particles blows up as their size goes to zero and particles are allowed to get too small too quickly. The unbounded fragmentation rate can result in a runaway fragmentation process and a loss of mass unaccounted for in the model formulation. This loss of mass was observed by McGrady and Ziff in \cite{ziff87}, the process was termed `shattering' and attributed to the creation of `dust' particles with zero size but positive mass. \\

\noindent In \cite{huang96}, Huang \textit{et al}. suggest that such a runaway fragmentation process is unphysical, and that at some point particles become too small to break-up any further. Their proposed model includes a cut-off size $x_c>0$, above which particles are able to fragment as usual. However, once a particle's size drops below $x_c$ it ceases to be able to fragment, becoming dormant. In this paper we shall present a variation on this theme, introducing a mixed discrete--continuous model as a solution to the problem of shattering. \\

\noindent Considering the nature of the material that is undergoing fragmentation  that we are attempting to model, on close inspection we might expect there to be some minimum fundamental unit (monomer) from which all particles are built up. Whatever level this occurs at, once this is imposed, the runaway fragmentation, associated with shattering, is prohibited. Such a framework necessarily induces a discrete nature on the material and suggests the use of a discrete model. However, on larger scales where a typical particle is composed of a large number of such monomers and hence where the mass of a particle is highly divisible, the continuum model may provide an adequate and convenient representation. With our hybrid model we attempt to reconcile these factors. \\

\noindent In our model, the smaller particles are considered to be comprised of collections of monomers. Through suitable scaling, the monomers can be assumed to have unit mass and therefore the smaller particles take positive integer mass, up to some cut-off value $N\in \mathbb{N}$. However, above this cut-off particle mass is considered as a continuous variable. A set-up such as this produces a dual regime model, with a discrete mass regime below the cut-off and a continuous mass regime above.\\

\noindent Let us denote by $u_C(x,t)$ the particle mass density within the continuous mass regime. The evolution of $u_C(x,t)$ is then governed by the continuous multiple fragmentation equation below:
\begin{align}\label{equation301}
\hspace{-8mm}\frac{\partial u_C(x,t)}{\partial t}&=-a(x)u_C(x,t)+\int_{x}^{\infty}a(y)b(x|y)u_C(y,t)\dd y, \hspace{1.92mm}x>N,\hspace{1.2mm}t>0,\\
                            u_C(x,0)&=c_{0}(x).\nonumber
\end{align}
\noindent This equation is of a similar form to that introduced in \cite{ziff87}. As in that model, the function $a(x)$ provides the fragmentation rate for a particle of mass $x$, whilst $b(x|y)$ represents the distribution of particles of mass $x>N$ resulting from the break-up of a particle of mass $y>x$. The functions $a$ and $b$ are assumed to be nonnegative measurable functions, defined on $\left(N,\infty\right)$ and $\left(N,\infty\right)\times\left(N,\infty\right)$, respectively. We also require $b(x|y)=0$ for $x>y$, since no particle resulting from a fragmentation event can have a mass exceeding the original parent particle. Finally, $c_{0}(x)$ details the initial mass distribution within the continuous regime.\\

\noindent The first term on the right-hand side of equation~\eqref{equation301} is a loss term; accounting for those particles of mass $x>N$ which are lost due to their fragmentation into smaller particles. The second term, involving the integral, is a gain term and corresponds to the increase we see in particles of mass $x$, due to the break-up of larger particles.\\

\noindent Turning our attention to the discrete mass regime, let $u_{Di}(t)$ denote the concentration of $i-$mer particles and $u_{D}(t)$ the $N$-vector taking these values as entries. The change in the values $u_{Di}(t)$, $i=1,\dots,N$, is governed by the equations:
\begin{align}\label{equation302}
\hspace{-80mm}\frac{\ddn u_{Di}(t)}{\ddn t} &\hspace{-.7mm}=\hspace{-.7mm}-a_iu_{Di}(t)+\hspace{-1.9mm}\sum_{j=i+1}^{N}\hspace{-1.2mm}a_jb_{i,j}u_{Dj}(t)
+\hspace{-1mm}\int_{N}^{\infty}\hspace{-2mm}a(y)b_i(y)u_{\small{C}}(y,t)\dd y,\hspace{.7mm}\hspace{1.1mm}t>0,\\
 u_{D}(0)&=d_{0}.\nonumber
\end{align}
\noindent In the case of $i=N$, the second term becomes an empty sum and is taken to be $0$. The values $a_i$ give the rates at which $i-$mer particles fragment, with $a_1=0$. The quantities $b_{i,j}$ give the expected number of $i-$mers produced from the fragmentation of a $j-$mer and the functions $b_i(y)$ give the expected number of $i-$mers produced from the fragmentation of a particle of mass $y>N$. The underlying physics demands that each $a_i$, $b_{i,j}$ and $b_i(y)$ be nonnegative. Finally, $d_{0}$ is the $N$-vector giving the  initial mass distribution within the discrete regime.\\

\noindent Analogously to equation~\eqref{equation301}, the first term on the right-hand side of equation~\eqref{equation302} is a loss term, accounting for the loss in $i-$mer particles due to their fragmentation into smaller particles. The remaining two terms are gain terms, with the term involving the summation giving the increase in $i-$mers due to the break-up of larger $j-$mers and the integral term representing production of new $i-$mers from the fragmentation of larger continuous mass particles.\\

\noindent In any fragmentation event, mass is simply redistributed from the larger particle to the smaller resulting particles, but the total mass involved should be conserved. This gives us the following two mass conservation conditions to supplement equations~\eqref{equation301} and~\eqref{equation302}:
\begin{align}
\label{equation303}&\int_{N}^{y}xb(x|y)\dd x+\sum_{j=1}^{N}jb_j(y)=y\hspace{2mm} \text{for} \hspace{2mm}y>N,\\
\label{equation304}&\sum_{j=1}^{i-1}jb_{j,i}=i\hspace{2mm} \text{for} \hspace{2mm}i=2,\ldots,N.
\end{align}
The condition~\eqref{equation303} is an expression of mass conservation upon the fragmentation of a particle from the continuous mass regime. The integral term gives the expected mass accounted for by resulting particles remaining within the continuous mass regime, that is those with mass lying in the range $N<x<y$, whereas the summation term represents the expected total mass attributable to the resulting particles in the discrete mass regime, i.e. those taking an integer value from $1$ to $N$. The equation~\eqref{equation304} comes from the conservation of mass when a particle from the discrete mass regime breaks up. Only one term is required for this condition as when a particle of discrete mass fragments, all resulting particles must themselves lie within the discrete mass regime.\\

\noindent The necessity of these conditions can be seen from equations~\eqref{equation301} and~\eqref{equation302}. If we integrate the right-hand side of~\eqref{equation301} over $(N,\infty)$ with respect to the measure $x\dd x$ and if we multiply the right-hand side of~\eqref{equation302} by $i$ and then sum over $i$ from $1$ to $N$, then formally, the sum of these two quantities gives us the rate of change of the total mass. Equating the continuous and discrete components of the resulting expression to zero, provides us with the conditions~\eqref{equation303} and~\eqref{equation304} respectively. However, these conditions alone are insufficient to guarantee mass conservation since the validity of the associated calculation requires a degree of regularity from the solutions, which is not known a priori.

\section{Preliminaries}
\noindent In the analysis of our equations we shall be relying heavily on the methods and theory of operator semigroups. In particular the concept of substochastic semigroups, the Kato--Voigt perturbation theorem and the notion of semigroup honesty. Additionally, in order to handle our system of equations we shall apply results concerning operator matrices acting on product spaces, and the semigroups they generate. For the sake of completeness we include here a rundown of the most significant results for our purposes.

\begin{definition}\label{definition262}Let $X$ denote a Banach space of the type $L_1(\Omega,\mu)$ with positive cone $X_+$, where $\Omega$ is a measurable subset of $\mathbb{R}^n$ and $\mu$ is a nonnegative measure. Additionally, let $(T(t))_{t\geq0}$ be a $C_{0}$-semigroup on $X$. We say that $(T(t))_{t\geq0}$ is a \textit{substochastic} semigroup on $X$ if, for each $t\geq0$, $\|T(t)\|\leq1$ and $T(t)f\in X_+$ for all $f\in X_+$. If additionally $\|T(t)f\|=\|f\|$ for all $t\geq0$ when $f\in X_+$, then we say that $(T(t))_{t\geq0}$ is a \textit{stochastic} semigroup.
\end{definition}

\noindent When formulating our equation of interest as an abstract Cauchy problem, it is common that the terms which appear are more naturally expressed as the sum of two or more separate operators, perhaps due to the differing nature of the effects they are representing. Very often checking the conditions of the Hille--Yosida theorem directly for the sum would prove intractable. In situations such as this, it is often easier to consider the operators individually, making use of a set of theorems known as \textit{perturbation results}. In most of these results it is assumed that one of the individual operators, generates a $C_{0}$-semigroup. The question then arises under what conditions on the other operator the combined operator sum (or some related operator) forms a generator of a $C_{0}$-semigroup.

\begin{theorem} \label{theorem285} Let the linear operator $(A,D(A))$ generate a $C_{0}$-semigroup $(T(t))_{t\geq0}$, on a Banach space $X$, satisfying the standard bound
\[\|T(t)\| \leq M{\mathrm{e}}^{\omega t} \hspace{3mm} \textup{for} \hspace{3mm} t\geq0,\]
for some $M>0$ and $\omega\geq0$.
If $B\in B(X)$, that is $B$ is a bounded linear operator from $X$ into $X$, then the sum $A+B$ with $D(A+B)=D(A)$ generates a $C_{0}$-semigroup, $(S(t))_{t\geq0}$, satisfying
\[\|S(t)\| \leq M{\mathrm{e}}^{(\omega+M\|B\| )t} \hspace{3mm} \textup{for} \hspace{3mm} t\geq0.\]
\end{theorem}
\begin{proof}See \cite[Chapter 3, Theorem 1.3]{engel00}.
\end{proof}

\noindent For some of the upcoming applications, the requirement that $B$ be bounded will turn out to be too restrictive. We therefore turn to an alternative perturbation result, namely the Kato--Voigt perturbation theorem. This result does not rely on $B$ being bounded. However, in removing this restriction we lose $A+B$ as our generator and instead we can only say that some extension of $A+B$ is a generator.

\begin{theorem} \label{theorem288} \textnormal{(Kato--Voigt Perturbation Theorem)} Let $X=L_1(\Omega,\mu)$ and suppose the linear operators $A$ and $B$, acting on $X$, satisfy the conditions:
\begin{enumerate}
 \item $(A,D(A))$ generates a substochastic semigroup $(G_A(t))_{t\geq0}$ on X;
 \item $B$ is a positive linear operator, that is $B:D(B)_+\mapsto X_+$, with domain satisfying $D(A)\subseteq D(B)$;
 \item For all $f\in D(A)_+$,
 \[\int_{\Omega}(Af+Bf)\dd \mu\leq 0.\]
\end{enumerate}
Then, there exists an extension $(K,D(K))$ of the operator $(A+B,D(A))$, which generates a substochastic semigroup $\left(G_K(t)\right)_{t\geq0}$.
\end{theorem}
\begin{proof}See \cite[Corollary 5.17]{banasiak06Pert}.
\end{proof}

\begin{remark}\label{remark283}
\noindent For reasons which will become apparent in the upcoming definition, it is common to express condition $(iii)$ in the form
\begin{equation} \label{equation2872}
\int_{\Omega}(A+B)f\dd \mu=-c(f)\hspace{2mm}\text{for}\hspace{2mm}f\in D(A)_+,
\end{equation}
where $c$ is some nonnegative linear functional defined on $D(A)$.
\end{remark}

\noindent Theorem~\ref{theorem288} was first applied in field of fragmentation equations by Banasiak in \cite{banasiak01}, where a particular case of the multiple fragmentation equation was examined, and more generally by Lamb \cite{lamb04} and Banasiak and Arlotti \cite{banasiak06Pert} to establish the existence of unique mass-conserving positive solutions  under suitable constraints on the fragmentation rate. This approach has proved particularly fruitful and has been applied to a range of coagulation--fragmentation models, for example in \cite{banasiak03, banasiak09, blair07, smith10, smith12}. However, a practical downside of this result is that it guarantees only the existence of a generator $K$, and provides no indication of how this operator relates to $A+B$. The nature of the generator $K$ is closely related to the concept of semigroup honesty, which we now define.

\begin{definition}\label{definition286}The positive semigroup $\left(G_K(t)\right)_{t\geq0}$, generated by the extension $K$ of $A+B$ from Theorem~\ref{theorem288}, is \emph{honest} if the functional linear $c$, given by \eqref{equation2872}, extends to $D(K)$, and for all $u_{0}\in D(K)_+$, the nonnegative solution $u(t)=G_K(t)u_{0}$ to
\[\frac{\ddn}{\ddn t}u(t)=Ku(t),\hspace{3mm}t>0;\hspace{3mm}u(0)=u_0,\]
 satisfies
\vspace{3mm}
\[\frac{\ddn}{\ddn t} \left\|u(t)\right\|
=\frac{\ddn}{\ddn t}\int_{\Omega}u(t)\dd \mu=-c(u(t)),\]
\end{definition}
\noindent where $\left\|\cdot\right\|$ is the norm of the space $x$ from Theorem~\ref{theorem288}. The following result provides necessary and sufficient conditions on the generator $K$ such that the related semigroup is honest.
\begin{theorem}\label{theorem289} The semigroup $\left(G_K(t)\right)_{t\geq0}$ is honest if and only if $K=\overline{A+B}$, where $\overline{A+B}$ denotes the closure of $A+B$.
\end{theorem}
\begin{proof}See \cite[Theorem 6.13]{banasiak06Pert}.
\end{proof}

\noindent In the upcoming analysis we shall rely on results which allow us to establish this condition in practice and also explicitly obtain the generator $K$. However, their explanation is heavily dependent on the specific application and involves material which is not suitable for this section. Therefore, we leave the introduction of the aforementioned results until later, where they appear as Theorem~\ref{theorem302} and Lemma~\ref{lemma3002}.\\

\noindent The mixed discrete--continuous fragmentation model introduced above involves two equations, describing quantities which are fundamentally different in nature. When looking to reformulate these equations we find that the differing nature of the equations means that different spaces are best suited for their analysis. However, just as it is possible to express a system of $n$ scalar differential equations as a single equation in $\mathbb{R}^n$ using matrix notation, we may transform our system of abstract equations into a single abstract Cauchy problem. The underlying space is now a product space and the (generating) operator takes the form of a matrix whose entries are themselves operators which map from and to the relevant spaces. In our case, the specific nature of the problem means that the matrix in question will be a 2$\times$2 matrix of upper triangular form. The upcoming  Theorem~\ref{theorem290} gives sufficient conditions for such an operator to be a generator, as well as providing the semigroup generated. However, before we can outline the conditions of Theorem~\ref{theorem290}, we require one further definition and an associated result which we shall utilise when we later come to apply Theorem~\ref{theorem290}.

\begin{definition}\label{definition237} Let $(\hspace{.12mm}X,\hspace{0.25mm}\|\cdot\|_X\hspace{.15mm})$ \hspace{-1.6mm} and \hspace{-1.4mm} $(\hspace{.15mm}Y,\hspace{0.25mm}\|\cdot\|_Y\hspace{.15mm})$ be Banach spaces and let $A:D(A)\subseteq X  \rightarrow X$ and $B:D(B)\subseteq X \rightarrow Y$ be linear operators with $D(A)\subseteq D(B)$. We say that $B$ is $A$-\textit{bounded} \textup{(}or $B$ is \textit{relatively} $A$-\textit{bounded}\textup{)} if there exist nonnegative constants $a$ and $b$ such that
\begin{equation}\label{equation20909}
\|Bf\|_Y\leq a\|Af\|_X+b\|f\|_X\hspace{2mm}\text{for all}\hspace{2mm}f\in D(A).
\end{equation}
\noindent The infimum of the values of $a$ for which such a bound exists is known as the $A$-\textit{bound of} $B$.
\end{definition}

\begin{lemma}\label{lemma231} Let $(X,\|\cdot\|_X)$ and $(Y,\|\cdot\|_Y)$ be Banach spaces and suppose the linear operators $A:D(A)\subseteq X\rightarrow X$ and $B:D(B)\subseteq X \rightarrow Y$ have domains satisfying $D(A)\subseteq D(B)$, with $A$ additionally having a nonempty resolvent set $\rho(A)$. Then $B$ is $A$-bounded if and only if $BR(\lambda,A)\in B(X,Y)$ for some $\lambda \in \rho(A)$, where $B(X,Y)$ denotes the set of bounded linear operators from $X$ into $Y$.
\end{lemma}
\begin{proof}See \cite[Lemma 4.1]{banasiak06Pert}.
\end{proof}
\noindent Having defined the concept of the relative boundedness of operators, we are able to detail the conditions which are sufficient to guarantee that our operator matrix generates a semigroup on the associated product space.
\begin{theorem} \label{theorem290} Let $X$ and $Y$ be Banach spaces. Consider the operator matrix
\[\textit{\textbf{A}}\hspace{-.5mm}=\hspace{-.5mm}\left(\hspace{-1mm}\begin{array}{cc}A & B \\0 & D \\ \end{array}\hspace{-1mm}\right),\]
and suppose that the following hold for the linear operators $A$, $B$ and $D$:
\begin{enumerate}
  \item $A:D(A)\subseteq X\rightarrow X$ generates a $C_{0}$-semigroup $(T(t))_{t\geq0}$ on $X$;
  \item $D:D(D)\subseteq Y\rightarrow Y$ generates a $C_{0}$-semigroup $(S(t))_{t\geq0}$ on $Y$;
  \item $B:D(B)\subseteq Y\rightarrow X$ is relatively $D$-bounded;
  \item $(\textit{\textbf{A}},D(\textit{\textbf{A}}))$ is a closed operator;
  \item the operator $\tilde{R}(t):D(D)\subseteq Y\rightarrow X$ given by $\tilde{R}(t)f=\int_{0}^{t}T(t-s)BS(s)f \dd s$, has a unique
        extension $R(t)\in B(Y,X)$ which is uniformly bounded as $t\searrow0$.
\end{enumerate}
Then \textit{\textbf{A}}, with domain $D(\textit{\textbf{A}})=D(A)\times D(D)\subseteq X\times Y$, generates a strongly continuous semigroup $\left(\textit{\textbf{T}}(t)\right)_{t\geq0}$ on the product space
$X\times Y$. Moreover, this semigroup is given by
\[\textit{\textbf{T}}(t):=\left(\hspace{-1mm}\begin{array}{cc}T(t) & R(t) \\0 & S(t) \\\end{array}\hspace{-1mm}\right), \hspace{2mm}t\geq0.\]
\end{theorem}
\begin{proof}See \cite[Proposition 3.1]{nagel89}.
\end{proof}
\noindent Having covered the requisite results from the theory of operator semigroups we are now in a position to commence the analysis of our model.

\section{Continuous Fragmentation Regime}
\noindent Looking initially at equation~\eqref{equation301}, we shall conduct our analysis of this equation within the setting of the weighted Lebesgue space $X_C=L_{1}\left((N,\infty),x\dd x\right)$. This is an obvious choice of space in which to study the problem, as the norm $\left\|\cdot\right\|_{X_C}$, when applied to the particle mass density $u_C$, provides a measure of mass. From the terms of equation~\eqref{equation301}, we introduce the following expressions
\begin{equation*}
(\mathcal{A}f)(x)=-a(x)f(x)\hspace{2mm} \text{and}\hspace{2mm} (\mathcal{B}f)(x)=\int_{x}^{\infty}a(y)b(x|y)f(y)\dd y\hspace{2mm} \text{for} \hspace{2mm}x>N.
\end{equation*}
\noindent From these expressions we form the  operators $A_C$ and $B_C$ as follows:
\[(A_Cf)(x)=(\mathcal{A}f)(x), \hspace{6mm}D(A_C)=\left\lbrace f\in X_C:A_Cf \in X_C\right\rbrace,\]
\[(B_Cf)(x)=(\mathcal{B}f)(x), \hspace{6mm}D(B_C)=\left\lbrace f\in X_C:B_Cf \in X_C\right\rbrace.\]

\noindent The following result relates the given domains of these operators, allowing us to consider taking their
sum $A_C+B_C$.

\begin{lemma} \label{lemma301}$D(A_C)\subseteq D(B_C)$ as $\left\|B_Cu\right\|_{X_C}\leq\left\|A_Cu\right\|_{X_C}$ for $u\in D(A_C)$. Hence $\left(A_C+B_C,D(A_C)\right)$ is a well-defined operator.
\end{lemma}

\begin{proof}Let $f\in D(A_C)$. Then
\begin{align}\label{equation305}
\left\|B_Cf\right\|_{X_C}&=\int_{N}^{\infty}\left|\int_{x}^{\infty}a(y)b(x|y)f(y)\dd y\right|x\dd x\nonumber\\
&\leq\int_{N}^{\infty}\left(\int_{x}^{\infty}a(y)b(x|y)\left|f(y)\right|\dd y\right)x\dd x\nonumber\\
&=\int_{N}^{\infty}a(y)\left|f(y)\right|\left(\int_{N}^{y}xb(x|y)\dd x\right)\dd y\\
&\leq \int_{N}^{\infty}a(y)\left|f(y)\right|y\dd y=\left\|A_Cf\right\|_{X_C}.\nonumber
\end{align}\\
\noindent Hence we have $f\in D(B_C)$, and so $D(A_C)\subseteq D(B_C)$. The final inequality follows as $\int_{N}^{y}xb(x|y)\dd x\leq y$, due to the mass conservation condition~\eqref{equation303}. This reflects the fact that upon fragmentation of a particle of mass $y>N$, the total mass of the resulting particles remaining within the continuous regime cannot exceed $y$.
\end{proof}
\noindent This allows us to form the operator $A_C+B_C$ with domain $D(A_C)$. Equation~\eqref{equation301} is then reformulated in the setting of $X_C$ as the abstract Cauchy problem:
\vspace{1mm}
\begin{equation}\label{equation306}
\frac{\ddn}{\ddn t}u_C(t)=K[u_C(t)],\hspace{3mm}t>0;\hspace{3mm}u_C(0)=c_{0}\in D(K),
\end{equation}
\noindent where $K$ is some extension of the operator $A_C+B_C$. The Kato--Voigt perturbation theorem (Theorem~\ref{theorem288}) will allow us to prove the existence of such an operator $K$, which generates a semigroup.

\begin{theorem} \label{theorem301}There exists an extension $(K,D(K))$ of $(A_C+B_C,D(A_C))$, which generates a substochastic semigroup $\left(G_K(t)\right)_{t\geq0}$.
\end{theorem}

\begin{proof}To establish this result we show that the three conditions set out in Theorem~\ref{theorem288} are satisfied for our particular operators $A_C$ and $B_C$.

\begin{enumerate}
  \item It is clear that $\left(A_C,D(A_C)\right)$ generates a substochastic semigroup\\ $\left(G_{A_C}(t)\right)_{t\geq0}$ on $X_C$, where $\left(G_{A_C}(t)f\right)(x)=\exp(-a(x)t)f(x)$, for $f\in X_C$.
  \item We have shown in Lemma~\ref{lemma301} that $D(A_C)\subseteq D(B_C)$. The nonnegativity of $a$ and $b$ imply that $B_C$ a positive operator, so that $B_Cf\in X_{C+}$ for all $f\in D(B_C)_+$.
  \item For all $f\in D(A_C)_+$ we have that
  \begin{align*}
   \hspace{-4mm}&\int_{N}^{\infty}\hspace{-1.5mm}\left(A_Cf+B_Cf\right)x\dd x=\hspace{-.5mm}\int_{N}^{\infty}\hspace{-1.5mm}
   \left(-a(x)f(x)+\int_{x}^{\infty}\hspace{-1.5mm}a(y)b(x|y)f(y)\dd y\right)x\dd x\\
   &=-\int_{N}^{\infty}a(x)f(x)x\dd x+\int_{N}^{\infty}\left(\int_{x}^{\infty}a(y)b(x|y)f(y)\dd y\right)x\dd x\\
   &=-\int_{N}^{\infty}a(x)f(x)x\dd x+\int_{N}^{\infty}a(y)f(y)\left(\int_{N}^{y}xb(x|y)\dd x\right)\dd y\\
   &=-\int_{N}^{\infty}\left(x-\int_{N}^{x}yb(y|x)\dd y\right)a(x)f(x)\dd x=:-c(f)\leq0.\\
  \end{align*}
\end{enumerate}
\vspace{-4mm}
We have introduced the notation $c$ to represent the final integral expression, and this functional will have significance in the analysis which follows. The nonnegativity of $c$ comes as a result of the earlier statement regarding $\int_{N}^{x}yb(y|x)\dd y\leq x$. The conditions of Theorem~\ref{theorem288} have been shown to hold in our case; hence there exists an extension $(K,D(K))$ of $(A_C+B_C,D(A_C))$, which generates a substochastic semigroup $\left(G_K(t)\right)_{t\geq0}$.
\end{proof}

\noindent This theorem proves only the existence of a generating extension $K$, and offers no indication of how exactly $K$ relates to $A_C+B_C$. The nature of the generator $K$ is closely related to the concept of the honesty of the semigroup (Definition~\ref{definition286}) and in turn the occurrence of `shattering'. This relationship is discussed in \cite{banasiak06PhyD}, where a range of possibilities for $K$ are considered and it is shown that the cases in which shattering occurs coincide with those in which the semigroup generated by $K$ is dishonest. \\

\noindent In order to establish the honesty of the semigroup $\left(G_K(t)\right)_{t\geq0}$, we follow a similar approach to that taken in \cite[Section 6.3]{banasiak06Pert}. Let us denote by $\mathsf{E}$ the set of all measurable functions defined on $(N,\infty)$, which take values within the extended reals. By $\mathsf{E}_f$ we denote the subspace of $\mathsf{E}$ consisting of functions which are finite almost everywhere. We also introduce the set $\mathsf{F}\subset\mathsf{E}$, defined as follows. The function $f\in\mathsf{F}$, if and only if given a nonnegative, nondecreasing sequence $\left\lbrace f_n\right\rbrace_{n=1}^\infty \subset\mathsf{E}$, where $\sup_{n\in\mathbb{N}}f_n=|f|$, we have $\sup_{n\in\mathbb{N}}\left(I-A_C\right)^{-1}f_n\in X_C$.\\

\noindent Additionally, we place the following two requirements on the operator $B_C$ and its domain $D(B_C)$. Firstly
\begin{equation}\label{equation307}
f\in D(B_C)\hspace{2.5mm} \text{if and only if}\hspace{2.5mm}f_+,f_-\in D(B_C),
\end{equation}
\\
\noindent where $f_{+}=\max\left\lbrace f,0\right\rbrace$ and $f_{-}=-\min\left\lbrace f,0\right\rbrace$. Secondly, for any two nondecreasing sequences $\left\lbrace f_n\right\rbrace_{n=1}^\infty$ and $\left\lbrace g_n\right\rbrace_{n=1}^\infty$ in $D(B_C)_+$, we have that

\begin{equation}\label{equation308}
\sup_{n\in \mathbb{N}}f_n=\sup_{n\in \mathbb{N}}g_n\hspace{2.5mm} \text{implies}\hspace{2.5mm}\sup_{n\in \mathbb{N}}B_Cf_n=\sup_{n\in \mathbb{N}}B_Cg_n.
\end{equation}

\begin{lemma} \label{lemma302} With $B_C$ restricted to $D(A_C)$, $(B_C,D(A_C))$ satisfies the conditions \eqref{equation307} and \eqref{equation308}.
\end{lemma}

\begin{proof}Initially let us assume that both $f_+,f_-\in D(A_C)$. Then, writing $f$ as $f=f_+-f_-$ and using the
linearity of $A$ with the triangle inequality, we get that
\[\left\|A_Cf\right\|_{X_C}=\left\|A_Cf_+-A_Cf_-\right\|_{X_C}\leq\left\|A_Cf_+\right\|_{X_C}+\left\|A_Cf_-\right\|_{X_C}.\]
Therefore $f\in D(A_C)$ when $f_+,f_-\in D(A_C)$. Now conversely, suppose that $f\in D(A_C)$. Since $0\leq f_\pm \leq|f|$, we have
\[\left\|A_Cf_\pm\right\|_{X_C}=\int_{N}^{\infty}a(y)f_\pm(y)y\dd y\leq \int_{N}^{\infty}a(y)|f(y)|y\dd y=\left\|A_Cf\right\|_{X_C}.\]
Hence if $f\in D(A_C)$ then $f_+,f_-\in D(A_C)$. Taken together, these two results give us the first of our conditions \eqref{equation307}. The second condition, \eqref{equation308}, follows using Lebesgue's monotone convergence theorem, which gives us
\[\sup_{n\in \mathbb{N}}B_Cf_n=\mathcal{B}\sup_{n\in \mathbb{N}}f_n=\mathcal{B}\sup_{n\in \mathbb{N}}g_n=\sup_{n\in \mathbb{N}}B_Cg_n.\]
\noindent Therefore the operator $B_C$ satisfies both of our requirements when it is restricted to the domain $D(A_C)$, which from now on we shall assume unless otherwise stated.
\end{proof}
\noindent We are nearly in a position to demonstrate the honesty of the semigroup $\left(G_K(t)\right)_{t\geq0}$. However, before we can do so we are required to introduce some further notation and detail a result we had been holding off since the previous section.\\

\noindent In addition to the above defined sets, we also introduce $\mathsf{G}\subset\mathsf{E}$ as the set of all functions $f\in X_C$ such that if $\left\lbrace f_n\right\rbrace_{n=1}^\infty$ is a nondecreasing sequence of nonnegative functions in $D(A_C)$ such that $\sup_{n\in \mathbb{N}}f_n=|f|$, then $\sup_{n\in \mathbb{N}}B_Cf_n<\infty$ almost everywhere.\\

\noindent The final items of notation which we must introduce are the mappings $\mathsf{B}:D(\mathsf{B})_{+}\rightarrow\mathsf{E}_{f,+}$, where $D(\mathsf{B})=\mathsf{G}$ and $\mathsf{L}:\mathsf{F}_{+}\rightarrow X_{C+}$ defined by
\begin{align*}
&\mathsf{B}f:=\sup_{n\in \mathbb{N}}B_Cf_n, \hspace{10mm}f\in D(\mathsf{B})_{+},\\
&\mathsf{L}f:=\sup_{n\in \mathbb{N}}R(1,A_C)f_n, \hspace{7mm}f\in \mathsf{F}_{+},
\end{align*}
where $0\leq f_n\leq f_{n+1}$ for all $n\in \mathbb{N}$ and $\sup_{n\in \mathbb{N}}f_n=f$.\\

\noindent With the set notations and extension operators defined, we can now detail the key generator characterisation result, which will enable us to establish the honesty of our semigroup.
\begin{theorem} \label{theorem302} If for all $f\in\mathsf{F}_+$ such that $-f+\mathsf{BL}f\in X_C$ and $c(\mathsf{L}f)$ exists it is true that
\begin{equation}\label{equation309}
\int_{N}^{\infty}\mathsf{L}fx\dd x+\int_{N}^{\infty}\left(-f+\mathsf{BL}f\right)x\dd x\geq-c\left(\mathsf{L}f\right),
\end{equation}
then $K=\overline{A_C+B_C}$.
\end{theorem}
\begin{proof}See \cite[Theorem 6.22]{banasiak06Pert}.
\end{proof}

\begin{theorem}\label{theorem303} If the fragmentation rate, $a(x)$, is such that
\[\limsup_{x\rightarrow{N}^+}a(x)<\infty \hspace{2.5mm} \text{and}\hspace{2.5mm}a\in L_{\infty,loc}(N,\infty),\]
\noindent then the semigroup $\left(G_K(t)\right)_{t\geq0}$ is honest.
\end{theorem}
\begin{proof} The proof of this result follows closely that of \cite[Theorem 8.5]{banasiak06Pert}. Since $A_Cf=-af$, as in \cite[Corollary 3.1]{banasiak01}, we have that $\mathsf{F}=\left\lbrace f\in\mathsf{E}:(1+a)^{-1}f\in X_C\right\rbrace$ and $\mathsf{L}f=(1+a)^{-1}f$, whilst by Lebesgue's monotone convergence theorem the operator $\mathsf{B}$ is given by the integral expression $\mathcal{B}$. \\

\noindent For $f\in\mathsf{F}_+$, let $g=\mathsf{L}f=(1+a)^{-1}f\in X_{C+}$. Then we see that the condition \eqref{equation309} is satisfied if for all $g\in X_{C+}$ such that $-ag+\mathcal{B}g\in X_C$ and $c(g)$ exists, we have:
\vspace{3mm}
\begin{equation}\label{equation310}
\int_{N}^{\infty}\left(-a(x)g(x)+\left(\mathcal{B}g\right)(x)\right)x\dd x\geq-c\left(g\right).
\end{equation}
\noindent By our assumptions regarding the function $a$, we have $ag\in L_{1}\left((N,R],x\dd x\right)$ for any $N<R<\infty$ with $\mathcal{B}g\in L_{1}\left((N,R],x\dd x\right)$ also, since $-ag+\mathcal{B}g\in X_C$. We may write the left-hand side of \eqref{equation310} as
\begin{align}\label{equation311}
 &\int_{N}^{\infty}\left(-a(x)g(x)+\left(\mathcal{B}g\right)(x)\right)x\dd x=\lim_{R\rightarrow\infty}\int_{N}^{R}\left(-a(x)g(x)+\left(\mathcal{B}g\right)(x)\right)x\dd x\nonumber\\
&=\lim_{R\rightarrow\infty}\left\lbrace -\int_{N}^{R}a(x)g(x)x\dd x+\int_{N}^{R}\left(\int_{x}^{\infty}a(y)b(x|y)g(y)\dd y\right)x\dd x\right\rbrace.
\end{align}
\noindent If we take the second term from above, split the inner integral in two and then change the order of integration for the integral over the bounded domain, then we get
\begin{align*}
&\int_{N}^{R}\left(\int_{x}^{\infty}a(y)b(x|y)g(y)\dd y\right)x \dd x\\
&=\int_{N}^{R}\left(\int_{x}^{R}a(y)b(x|y)g(y)\dd y\right)x \dd x+\int_{N}^{R}\left(\int_{R}^{\infty}a(y)b(x|y)g(y)\dd y\right)x \dd x\\
&=\int_{N}^{R}a(y)g(y)\left(\int_{N}^{y}xb(x|y)\dd x\right)\dd y+\int_{N}^{R}\left(\int_{R}^{\infty}a(y)b(x|y)g(y)\dd y\right)x \dd x.
\end{align*}

\noindent Substituting this back into \eqref{equation311} then yields
\begin{align*}
&\int_{N}^{\infty}\left(-a(x)g(x)+\left(\mathcal{B}g\right)(x)\right)x\dd x\\
&=-\lim_{R\rightarrow\infty}\int_{N}^{R}\left(x-\int_{N}^{x}yb(y|x)\dd y\right)a(x)g(x) \dd x\\
&\hspace{5mm}+\lim_{R\rightarrow\infty}\int_{N}^{R}\left(\int_{R}^{\infty}a(y)b(x|y)g(y)\dd y\right)x \dd x\\
&=-c(g)+\lim_{R\rightarrow\infty}\int_{N}^{R}\left(\int_{R}^{\infty}a(y)b(x|y)g(y)\dd y\right)x \dd x.
\end{align*}
\noindent The nonnegativity of the additional term accompanying $-c(g)$ gives us \eqref{equation310}, and with that the honesty of the semigroup $\left(G_K(t)\right)_{t\geq0}$.
\end{proof}

\noindent Having obtained an $X_C$-valued solution to our abstract Cauchy problem equation \eqref{equation306}, we now examine whether this provides a scalar-valued solution to our original continuous regime equation \eqref{equation301}. But first we require the following result, which we have delayed until now as it concerns the extended operators introduced above.

\begin{lemma}\label{lemma3002}Suppose that $\left(A_C,D(A_C)\right)$ and $\left(B_C,D(B_C)\right)$ satisfy the conditions of Theorem~\ref{theorem288} along with conditions \eqref{equation307} and \eqref{equation308}. Additionally let us define the operator $\mathsf{T}$ with
$D(\mathsf{T})=\mathsf{LF} \subset X_C $ by
\[\mathsf{T}u=u-\mathsf{L}^{-1}u,\]
which is permitted since $\mathsf{L}$ is one-to-one \cite[Theorem 6.18]{banasiak06Pert}. Then, the extension $K$ of $A_C+B_C$ that generates a substochastic semigroup on $X_C$ is given by
\[Ku=\mathsf{T}u+\mathsf{B}u,\]
\noindent with
\[D(K)=\left\lbrace u\in D(\mathsf{T})\cap D(\mathsf{B}):\mathsf{T}u+\mathsf{B}u \in X_C, \text{  and  }\lim_{n\rightarrow\infty}\left\|(\mathsf{L}\mathsf{B})^nu \right\|=0 \right\rbrace.\]
\end{lemma}
\begin{proof}See \cite[Theorem 6.20]{banasiak06Pert}
\end{proof}

\begin{theorem}\label{theorem3003}There exists a measurable scalar-valued representation $u_C(x,t)$ of the semigroup solution $(G_K(t)c_0)(x)$, such that $u_C(x,t)$ is absolutely continuous with respect to $t$, the partial derivative $\partial_t u_C(x,t)$ exists almost everywhere on $(N,\infty)\times[0,\infty)$ and $u_C(x,t)$ satisfies equation \eqref{equation301} for almost all $x>N$ and $t>0$. Further, this representation is unique up to sets of measure zero.
\end{theorem}

\begin{proof}By \cite[Theorem 2.40]{banasiak06Pert}, since $G_K(t)c_0$ is continuously differentiable, there exists a real-valued function $u_C(x,t)$, measurable on $(N,\infty)\times[0,\infty)$, which is absolutely continuous with respect to $t$ for almost all $x\in(N,\infty)$ and such that $\partial_t u_C$ exists with $u_C(x,t)=\left(G_K(t)c_0\right)(x)$ and\\
\begin{equation}\label{equation3011}
\frac{\partial u_C(x,t)}{\partial t}=\left[\frac{\ddn}{\ddn t} G_K(t)c_0\right](x),
\end{equation}\\
\noindent for almost all $t\in [0,\infty)$ and $x\in(N,\infty)$. Furthermore, the representation is unique up to sets of measure zero. We noted in Theorem~\ref{theorem303} that the operator $\mathsf{L}$ is defined by
\[\left[\mathsf{L}f\right](x)=(1+a(x))^{-1}f(x).\]
Hence the operator $\mathsf{T}$ introduced in Lemma~\ref{lemma3002} is given here by
\begin{equation*}
\left[\mathsf{T}f\right](x)=f(x)-\left[\mathsf{L}^{-1}f\right](x)=f(x)-(1+a(x))f(x)=-a(x)f(x),
\end{equation*}
and therefore on the domain $D(\mathsf{T})$, the operator $\mathsf{T}$ agrees with $\mathcal{A}$. We have already stated in Theorem~\ref{theorem303} that the operator $\mathsf{B}$ is given by the integral expression $\mathcal{B}$. Hence Lemma~\ref{lemma3002} yields
\[Ku= \mathcal{A}u+\mathcal{B}u\hspace{2.5mm} \text{for}\hspace{2.5mm}u\in D(K).\]
Therefore, the semigroup solution $G_K(t)c_0$ satisfies
\begin{equation}\label{equation3013}
\left[\frac{\ddn}{\ddn t} G_K(t)c_0\right](x)=\left[\mathcal{A}G_K(t)c_0\right](x)+\left[\mathcal{B}G_K(t)c_0\right](x).
\end{equation}
The right-hand side of \eqref{equation3013} is independent of our choice of representation of $G_K(t)c_0$, up to sets of measure zero. Hence combining \eqref{equation3011} and \eqref{equation3013}, we get that $u_C(x,t)$ satisfies
\begin{align*}
\hspace{-8mm}\frac{\partial u_C(x,t)}{\partial t}&=-a(x)u_C(x,t)+\int_{x}^{\infty}a(y)b(x|y)u_C(y,t)\dd y,
\end{align*}
for almost all $x>N$ and $t>0$, as required.
\end{proof}

\section{Discrete Fragmentation Regime}
\noindent We now turn our attention to the discrete mass regime and equation \eqref{equation302}. As with
equation \eqref{equation301}, our intention is to recast the equations as an abstract differential equation
within an appropriate vector space. With this aim in mind, we introduce the space $X_D=\mathbb{R}^N$, equipped with the weighted norm:
\[\left\|v\right\|_{X_D}=\sum_{j=1}^{N}j|v_j|,\hspace{2mm} \text{where} \hspace{2mm}v=(v_1,\ldots,v_N).\]

\noindent The choice of this norm is driven by the fact that, when applied to our solution, it will provide a
measure of the total mass within the discrete regime. From the first two terms on the right-hand side of equation \eqref{equation302} we get the operators $A_D$ and $B_D$ defined on $X_D$ by \[(A_Dv)_i=-a_iv_{i}\hspace{2mm}\text{and}\hspace{2mm}(B_Dv)_i=\sum_{j=i+1}^{N}a_jb_{i,j}v_{j},\hspace{2mm}\text{for}\hspace{2mm}i=1,
\ldots,N,\]
where in the case of $i=N$, the empty sum in $(B_Dv)_N$ is taken to be 0. As linear operators on a finite-dimensional space, both $A_D$ and $B_D$ are bounded, hence by \cite[Theorem 1.2]{pazy83}, $A_D+B_D$ generates a uniformly continuous semigroup $\left(T(t)\right)_{t\geq0}$ on $X_D$.

\begin{lemma} \label{lemma303} The semigroup $\left(T(t)\right)_{t\geq0}$ generated by $A_D+B_D$ on $X_D$ is a positive semigroup.
\end{lemma}
\begin{proof}Let us consider the equation
\[\left(\lambda I_D-(A_D+B_D)\right)v=w,\]
with $\lambda>0$. Adopting the convention that sums over empty index sets are zero, the above equation written elementwise becomes
\[(\lambda+a_i)v_{i}-\hspace{-1.6mm}\sum_{j=i+1}^{N}a_jb_{i,j}v_{j}=w_i.\]
Inverting this for $i=N$ we get
\[(\lambda+a_N)v_{N}=w_N\Rightarrow v_{N}=\frac{w_N}{\lambda+a_N}.\]
Now suppose that for $i=N,N-1,\ldots,k+1,$ we have $v_i=F_i(\lambda,w)$ where $F_i\geq0$ for $\lambda>0$ and $w\in\mathbb{R}_+^N$. Then, for $i=k$, we have
\begin{align*}
(\lambda+a_k)v_{k}-\hspace{-1.6mm}\sum_{j=k+1}^{N}\hspace{-1.6mm}a_jb_{k,j}F_j(\lambda,w)=w_k,
\end{align*}
and therefore
\begin{align*}
v_{k}=\frac{1}{\lambda+a_k}\left(\hspace{-.3mm}w_k
+\hspace{-1.6mm}\sum_{j=k+1}^{N}\hspace{-1.6mm}a_jb_{k,j}F_j(\lambda,w)\hspace{-.3mm}\right)=F_k(\lambda,w)\geq0.
\end{align*}

\noindent Therefore, by induction the resolvent operator $R(\lambda,A_D+B_D)$ is a positive operator for $\lambda>0$. Hence by \cite[Chapter 6, Theorem 1.8]{engel00} the semigroup $\left(T(t)\right)_{t\geq0}$ is positive.
\end{proof}

\noindent From the third term of equation \eqref{equation302} we define the operator $C:D(C)\subseteq X_C\rightarrow X_D$, pointwise by
\[(Cf)_i=\int_{N}^{\infty}a(y)b_i(y)f(y)\dd y, \hspace{6mm} D(C)=\left\lbrace f\in X_C:Cf \in X_D\right\rbrace,\]
\noindent for $i = 1,\ldots,N$. Equation~\eqref{equation302} is then reformulated as the abstract equation:
\vspace{1mm}
\begin{equation}\label{equation312}
\frac{\ddn}{\ddn t}u_D(t)=(A_D+B_D)[u_D(t)]+C[u_C(t)],\hspace{3mm}t>0;\hspace{3mm}u_D(0)=d_{0}.
\end{equation}
\noindent Having introduced the operator $C$, we now spend some time examining it in more detail and establishing
the properties we will require in the upcoming section. We begin with the following lemma relating the domain
of $C$ with that of $A_C$ from the previous section.

\begin{lemma} \label{lemma304}For $f\in D(A_C)$, $\left\|Cf\right\|_{X_D}\leq\left\|A_Cf\right\|_{X_C}$. Hence $D(A_C)\subseteq D(C)$.
\end{lemma}
\vspace{1mm}
\begin{proof}Let $f\in D(A_C)$. Then
\begin{align}\label{equation313}
\left\|Cf\right\|_{X_D}&=\sum_{i=1}^{N}i\left|\int_{N}^{\infty}a(y)b_i(y)f(y)\dd y\right|\nonumber\\
                     &\leq\sum_{i=1}^{N}i\left(\int_{N}^{\infty}a(y)b_i(y)\left|f(y)\right|\dd y\right)\nonumber\\
                     &=\int_{N}^{\infty}a(y)\left|f(y)\right|\left(\sum_{i=1}^{N}ib_i(y)\right)\dd y\\
                     &\leq\int_{N}^{\infty}a(y)\left|f(y)\right|y\dd y=\left\|A_Cf\right\|_{X_C}.\nonumber
\end{align}
\noindent The final inequality is a consequence of the mass conservation condition~\eqref{equation303}.
\end{proof}
\noindent Having established that $C$ is well-defined on $D(A_C)$, the next lemma provides a bound on $\left\|Cf\right\|_{X_D}$ when $f\in D(A_C)$.
\begin{lemma} \label{lemma305} The operator $C$ is $(A_C+B_C)$-bounded on $D(A_C)$.
\end{lemma}
\begin{proof} Combining \eqref{equation305} and \eqref{equation313}, for $f\in D(A_C)$ we have
\begin{align*}
\left\|B_Cf\right\|_{X_C}+\left\|Cf\right\|_{X_D}\leq&\int_{N}^{\infty}a(y)\left|f(y)\right|\left(\int_{N}^{y}xb(x|y)\dd x\right)\dd y\\
& +\int_{N}^{\infty}a(y)\left|f(y)\right|\left(\sum_{i=1}^{N}ib_i(y)\right)\dd y\\
=&\int_{N}^{\infty}a(y)\left|f(y)\right|\left(\int_{N}^{y}xb(x|y)\dd x+\sum_{i=1}^{N}ib_i(y)\right)\dd y\\
=&\int_{N}^{\infty}a(y)\left|f(y)\right|y\dd y=\left\|A_Cf\right\|_{X_C}.
\end{align*}
\noindent Subtracting $\left\|B_Cf\right\|_{X_C}$ from both sides gives us
\begin{align}\label{bound}
\left\|Cf\right\|_{X_D}&\leq\left\|A_Cf\right\|_{X_C}-\left\|B_Cf\right\|_{X_C}=\left\|A_Cf\right\|_{X_C}-\left\|-B_Cf\right\|_{X_C}\nonumber\\
                     &\leq\left\|A_Cf-(-B_Cf)\right\|_{X_C}=\left\|(A_C+B_C)f\right\|_{X_C}.
\end{align}
\end{proof}

\begin{lemma} \label{lemma306} The operator $C$ can be extended to $D(K)$, with this extension being $K$-bounded on $D(K)$.
\end{lemma}
\begin{proof} Let $v\in D(K)$; then from \cite[page 166]{kato95}, since $(K,D(K))$ is the closure of $(A_C+B_C,D(A_C))$, there exists a sequence $\left\lbrace v_n\right\rbrace_{n=1}^\infty\subset D(A_C)$ such that $v_n\rightarrow v$ and $(A_C+B_C)v_n\rightarrow Kv$ in $X_C$. By the linearity of the operators and Lemma~\ref{lemma305}, we have that
\[\|Cv_m-Cv_n\|_{X_D}\leq\|(A_C+B_C)v_m-(A_C+B_C)v_n\|_{X_C}.\]
\noindent The sequence $\left\lbrace(A_C+B_C)v_n\right\rbrace_{n=1}^\infty$ is convergent in $X_C$; therefore it must also be a Cauchy sequence in $X_C$ and, by the above bound, $\left\lbrace Cv_n\right\rbrace_{n=1}^\infty$ must also be Cauchy in $X_D$. Since the space $X_D$ is complete, the sequence $\left\lbrace Cv_n\right\rbrace_{n=1}^\infty$ must necessarily converge to a limit, which we denote by $Cv$. Further, the limit $Cv$ is independent of the sequence $\left\lbrace v_n\right\rbrace_{n=1}^\infty$, as we will now demonstrate. Suppose that $\left\lbrace w_n\right\rbrace_{n=1}^\infty\subset D(A_C)$ shares the attributes of $\left\lbrace v_n\right\rbrace_{n=1}^\infty$; then we can write
\begin{align*}
&\left\|Cw_n-Cv\right\|_{X_D}\leq\left\|Cw_n-Cv_n\right\|_{X_D}+\left\|Cv_n-Cv\right\|_{X_D}\\
                          &\leq\left\|(A_C+B_C)w_n-(A_C+B_C)v_n\right\|_{X_C}+\left\|Cv_n-Cv\right\|_{X_D}\\
                          &\leq\left\|(A_C+B_C)w_n-Kv\right\|_{X_C}+\left\|Kv-(A_C+B_C)v_n\right\|_{X_C}+\left\|Cv_n-Cv\right\|_{X_D}.
\end{align*}
\noindent From this we may deduce that $\left\lbrace Cw_n\right\rbrace_{n=1}^\infty$ also converges to $Cv$. The $K$-boundedness of $C$ on $D(K)$ is obtained from the $(A_C+B_C)$-boundedness of $C$ by passing the limits through the norms in~\eqref{bound}.
\end{proof}

\section{Full System}
\noindent Having considered both of the regimes separately, we now combine the equations from the continuous regime \eqref{equation306} and the discrete regime \eqref{equation312}, writing them as the following abstract Cauchy problem on the product space $X=X_D\times X_C$:
\begin{equation}\label{equation314}
\frac{\ddn}{\ddn t}u(t)=\textit{\textbf{A}}[u(t)],\hspace{3mm}t>0; \hspace{3mm}u(0)=u_0\in D(\textit{\textbf{A}})=X_D\times D(K),
\end{equation}
\noindent where $u(t)$, $\textit{\textbf{A}}$ and $u_0$ are given by\\
\[u(t)=\left(\begin{array}{c} u_D(t) \\ u_C(t) \\ \end{array} \right), \hspace{1mm}
\textit{\textbf{A}}=\left(\begin{array}{cc} (A_D+B_D) & C \\ 0_{DC} & K \\ \end{array} \right), \hspace{1mm}
u_0=\left( \begin{array}{c} d_0 \\ c_0 \\ \end{array} \right).\]\\
The presence of the subscripts on the zero operators indicate the spaces they map from and to; for example, $0_{DC}$ maps from $X_D$ into $X_C$. Our task is now to prove that the operator $\textit{\textbf{A}}$ generates a semigroup on the space $X$. In order to more easily show this, we consider $(\textit{\textbf{A}},D(\textit{\textbf{A}}))$ as the sum of two operators $(\textit{\textbf{A}}_1,D(\textit{\textbf{A}}))$ and $(\textit{\textbf{A}}_2,D(\textit{\textbf{A}}))$, where

\[\textit{\textbf{A}}
=\underbrace{\left(\begin{array}{cc} 0_{DD} & C \\ 0_{DC} & K \\ \end{array} \right)}_{\textit{\textbf{A}}_1}
+\underbrace{\left(\begin{array}{cc} (A_D+B_D) & 0_{CD} \\ 0_{DC} & 0_{CC} \\ \end{array}\right)}_{\textit{\textbf{A}}_2}.\]

\noindent From the boundedness of $(A_D+B_D)$ it is easily seen that $(\textit{\textbf{A}}_2,D(\textit{\textbf{A}}))$ is a bounded operator. The strategy is then to show that $(\textit{\textbf{A}}_1,D(\textit{\textbf{A}}))$ generates a semigroup and then treat $\textit{\textbf{A}}_2$ as a bounded perturbation of this generator. The first of these steps is tackled in the following theorem.

\begin{theorem}\label{theorem304} The operator $\textit{\textbf{A}}_1$, as given above, generates a $C_0$-semigroup on the product space $X=X_D\times X_C$.
\end{theorem}

\begin{proof} In order to establish that $\textit{\textbf{A}}_1$ is a generator, we demonstrate that the conditions of Theorem~\ref{theorem290} are satisfied by our operator. Conditions $(i)$ to $(iii)$ of Theorem~\ref{theorem290} are either straightforward, or else have already been verified. The element $0_{DD}$ generates the identity semigroup on $X_D$, whilst $K$ is the generator of the substochastic semigroup $\left(G_K(t)\right)_{t\geq0}$, as was shown in Theorem~\ref{theorem301}. Additionally, the operator $C$ is $K$-bounded on $D(K)$ as was shown in Lemma~\ref{lemma306}.\\

\noindent Condition $(iv)$ requires that $(\textit{\textbf{A}}_1,D(\textit{\textbf{A}}))$ be a closed operator. To see this, let us write $(\textit{\textbf{A}}_1,D(\textit{\textbf{A}}))$ as
\[\textit{\textbf{A}}_1
=\underbrace{\left(\begin{array}{cc} 0_{DD} & 0_{CD} \\ 0_{DC} & K \\ \end{array} \right)}_{\textit{\textbf{K}}}
+\underbrace{\left(\begin{array}{cc} 0_{DD} & C \\ 0_{DC} & 0_{CC} \\ \end{array}\right)}_{\textit{\textbf{C}}},\]
with both $\textit{\textbf{K}}$ and $\textit{\textbf{C}}$ having domain $D(\textit{\textbf{A}})$. Since $(K,D(K))$ is the generator of a strongly continuous semigroup, by a standard result, \cite[Theorem 2.13]{morante98}, it must be a closed operator. It is not difficult to deduce from this that $(\textit{\textbf{K}},D(\textit{\textbf{A}}))$ must also be closed. \\

\noindent Let us now introduce a new norm on the space $X$, defined for $f={{f_D}\choose{f_C}}\in X$ by $\|f\|_\alpha=\alpha\|f_D\|_{X_D}+\|f_C\|_{X_C}$, where $0<\alpha<1$. It is a straightforward exercise to show that this new norm is equivalent to the standard norm on $X$, and therefore $(\textit{\textbf{K}},D(\textit{\textbf{A}}))$
is closed with respect to this new norm. The $K$-boundedness of $C$ from Lemma~\ref{lemma306} gives us
\[\|\textit{\textbf{C}}f\|_\alpha=\alpha\|Cf_C\|_{X_D}\leq\alpha\|Kf_C\|_{X_C}=\alpha\|\textit{\textbf{K}}f\|_\alpha.\]
Therefore, with respect to this norm, $\textit{\textbf{C}}$ is $\textit{\textbf{K}}-$bounded with $\textit{\textbf{K}}-$bound less than $1$. We may then deduce from Lemma~\ref{lemma306}, that the operator $(\textit{\textbf{A}}_1,D(\textit{\textbf{A}}))$ is closed with respect to the norm $\|\cdot\|_\alpha$ and as a result of their equivalence, is also closed under the standard norm on $X$.\\

\noindent To prove the final condition of Theorem~\ref{theorem290}, we introduce the mapping $\tilde{Q}(t):D(K)\subseteq X_C\rightarrow X_D$ defined by $\tilde{Q}(t)f=\int_{0}^{t}CG_K(s)f \dd s$. For $f\in D(K)$ we may write this as
\begin{align}\label{equation3999}
\tilde{Q}(t)f&=\int_{0}^{t}CG_K(s)f\dd s\nonumber\\
     &=\int_{0}^{t}C(\lambda I_C-K)^{-1}(\lambda I_C-K)G_K(s)f\dd s\hspace{5mm}(\lambda>0)\nonumber\\
     &=C(\lambda I_C-K)^{-1}\int_{0}^{t}(\lambda I_C-K)G_K(s)f\dd s\nonumber\\
     &=C(\lambda I_C-K)^{-1}\left\lbrace\lambda\int_{0}^{t}G_K(s)f\dd s-\int_{0}^{t}KG_K(s)f\dd s\right\rbrace\nonumber\\
     &=C(\lambda I_C-K)^{-1}\left\lbrace\lambda\int_{0}^{t}G_K(s)f\dd s-\int_{0}^{t}G_K(s)Kf\dd s\right\rbrace\nonumber\\
     &=C(\lambda I_C-K)^{-1}\left\lbrace\lambda\int_{0}^{t}G_K(s)f\dd s-G_K(t)f+f\right\rbrace.
\end{align}

\noindent The extraction of $C(\lambda I_C-K)^{-1}$ from within the integral is permitted as it is a bounded linear operator, owing to $C$ being $K$-bounded and Lemma~\ref{lemma231}. The switching of the generator $K$ and the semigroup operator $G_K(t)$ is a standard semigroup result, detailed in \cite[Theorem 2.12]{morante98}, as is the replacement of the second integral in the final step, which can be found in \cite[Chapter 2, Lemma 1.3]{engel00}. Recalling that $\left(G_K(t)\right)_{t\geq0}$ is a semigroup of contractions, the following norm bound is readily obtained from \eqref{equation3999}:
\begin{align}\label{equation315}
\|\tilde{Q}(t)f\|_{X_D}&\leq \|C(\lambda I_C-K)^{-1}\|\left\lbrace\lambda\hspace{-.5mm}\int_{0}^{t}\hspace{-.5mm}\|G_K(s)f\|_{X_C}\dd s\hspace{-.5mm}+\hspace{-.5mm}\|G_K(t)f\|_{X_C}\hspace{-.5mm}
+\hspace{-.5mm}\|f\|_{X_C}\right\rbrace\nonumber\\
&\leq \|C(\lambda I_C-K)^{-1}\|\left(\lambda t+2\right)\|f\|_{X_C}.
\end{align}
\noindent Therefore $\tilde{Q}(t)$ is bounded on $D(K)$. As a densely-defined, bounded linear operator, $\tilde{Q}(t)$ can be uniquely extended, via taking limits, to a bounded linear operator $Q(t)$ in $B(X_C,X_D)$. Further, by passing limits through the norms, the bound \eqref{equation315} holds for all $u\in X_C$ with $\tilde{Q}(t)$ replaced by $Q(t)$. As such, $Q(t)$ is uniformly bounded as $t\searrow0$. By Theorem~\ref{theorem290}, $\textit{\textbf{A}}_1$ generates a $C_0$-semigroup on the product space $X$.
\end{proof}

\noindent Having established that $(\textit{\textbf{A}}_1,D(\textit{\textbf{A}}))$ is a generator, we are now ready to
prove the same for the full operator $(\textit{\textbf{A}},D(\textit{\textbf{A}}))$.

\begin{theorem} \label{theorem305} The operator $(\textit{\textbf{A}},D(\textit{\textbf{A}}))$, generates a $C_0$-semigroup on $X=X_D\times X_C$. Furthermore, this semigroup is given by\vspace{2mm}
\begin{equation}\label{matrixsemigroup}
\textit{\textbf{T}}(t):=\left(\begin{array}{cc}T(t) & R(t) \\0_{DC} & G_K(t) \\\end{array}\right), \hspace{2mm}t\geq0,
\end{equation}
where $R(t):X_C\rightarrow X_D$ is the unique bounded linear extension of the operator $\tilde{R}(t):D(K)\subseteq X_C\rightarrow X_D$ defined by
$\tilde{R}(t)f=\int_{0}^{t}T(t-s)CG_K(s)f\dd s$.
\end{theorem}

\begin{proof}As the sum of the generator $(\textit{\textbf{A}}_1,D(\textit{\textbf{A}}))$ and the bounded linear operator $(\textit{\textbf{A}}_2,D(\textit{\textbf{A}}))$, $(\textit{\textbf{A}},D(\textit{\textbf{A}}))$ is itself a generator by Theorem~\ref{theorem285}. The form of the semigroup comes as a consequence of Theorem~\ref{theorem290}.
\end{proof}
\noindent To summarise the results so far, the existence of the semigroup $\left(\textit{\textbf{T}}(t)\right)_{t\geq0}$, means that given an initial state $\binom{d_0}{c_0}\in D(\textit{\textbf{A}})$ and provided the conditions of Theorem~\ref{theorem303} are met, namely that the fragmentation rate $a(x)$ satisfies
\[\limsup_{x\rightarrow{N}^+}a(x)<\infty \hspace{2.5mm} \text{and}\hspace{2.5mm}a\in L_{\infty,loc}(N,\infty),\]
then there exists a unique solution to the system \eqref{equation306} and \eqref{equation312}. This solution is strongly differentiable with respect to $t$ and from \eqref{matrixsemigroup}, is given by
\begin{align}\label{semigroupsols}
u_D(t)&=T(t)d_0+\int_{0}^{t}T(t-s)CG_K(s)c_0\dd s=T(t)d_0+\int_{0}^{t}T(t-s)Cu_C(s)\dd s,\nonumber\\
u_C(t)&=G_K(t)c_0.
\end{align}
Furthermore, by Theorem~\ref{theorem3003}, the existence of the semigroup solution $u_C(t)=G_K(t)c_0$ to \eqref{equation306}, provides us with a unique (up to sets of measure zero) scalar-valued, classical solution to our original continuous equation \eqref{equation301}. The nature of the space $X_D$ and the strong derivative within this space, means that the solution $u_D(t)$ to
\eqref{equation312} automatically provides us with a set of unique classical solutions to the equations \eqref{equation302}. Having determined the existence of such solutions, in the upcoming section we establish some key properties displayed by them.

\section{Solution Properties (Nonnegativity and Mass Conservation)}
\noindent The existence of the $C_{0}$-semigroup $\left(\textit{\textbf{T}}(t)\right)_{t\geq0}$ provides us with a unique strong solution, $u(t)=\textit{\textbf{T}}(t)u_0$, to the abstract Cauchy problem \eqref{equation314}. For the solution to be physically relevant given the problem setting, we would expect it to possess a number of properties. In particular, we hope that the solution preserves nonnegativity and we demonstrate conservation of mass. First off we establish the nonnegativity of the strong solution.\\

\begin{lemma}\label{lemma307} Provided that the initial data are nonnegative, i.e. $u_0=\binom{d_0}{c_0}\in D(\textit{\textbf{A}})_{+}=X_{D+}\times D(K)_{+}$, then the strong solution $u(t)=\textit{\textbf{T}}(t)u_0$, emanating from $u_0$, remains nonnegative, that is $u(t)\in D(\textit{\textbf{A}})_{+}$ for all $t\geq0$.
\end{lemma}

\begin{proof} Let $u_0\in D(\textit{\textbf{A}})_{+}$. Then the two components of the solution $u(t)$ are given by  \eqref{semigroupsols}. Since $\left(G_K(t)\right)_{t\geq0}$ is a substochastic semigroup, it preserves nonnegativity of the initial datum. Therefore, when $c_0\in D(K)_{+}$ we have $G_K(t)f_0\in D(K)_{+}$ for all $t\geq0$, hence the continuous regime solution satisfies $u_C(t)\in D(K)_{+}$ for all $t\geq0$.\\

\noindent As an integral operator with a nonnegative kernel, the operator $C$ is easily seen to be a positive operator. The semigroup $\left(T(t)\right)_{t\geq0}$ is a positive semigroup as was shown in Lemma~\ref{lemma303}.
Together with the nonnegativity of the continuous regime solution, these ensure that the integral term is a nonnegative contribution to $u_D(t)$ whenever $c_0\in D(K)_{+}$. As stated, $\left(T(t)\right)_{t\geq0}$ is a positive semigroup, and therefore $u_D(t)\in X_{D+}$ for all $t\geq0$ when the initial system state satisfies $u_0=\binom{d_0}{c_0}\in X_{D+}\times D(K)_{+}$.\\

\noindent Together, having  $u_D(t)\in X_{D+}$ and $u_C(t)\in D(K)_{+}$ gives us $u(t)\in D(\textit{\textbf{A}})_{+}$ as required.
\end{proof}

\noindent Let $u(t)=\binom{u_D(t)}{u_C(t)}$ be the solution to the abstract Cauchy problem \eqref{equation314} with nonnegative initial data. Then the masses in the discrete and continuous regimes, at time $t$, are given by

\[M_D(t)=\sum_{i=1}^{N}iu_{Di}(t)\hspace{2mm}\text{and}\hspace{2mm}M_C(t)=\int_{N}^{\infty}\left(u_C(t)\right)(x)\,x\dd x,\]
\noindent respectively. Summing these gives us the total mass in the system $M(t)=M_D(t)+M_C(t)$ at time $t$. In each fragmentation event, mass is redistributed from larger particles to the smaller resulting particles, but the total mass involved should be conserved. This mass-conservation was built into our model in the form of the conditions \eqref{equation303} and \eqref{equation304}. Therefore, although $M_D(t)$ may increase and $M_C(t)$
decrease as larger particles break into smaller pieces, we would expect the total mass, $M(t)$, to remain constant.
\begin{lemma}\label{lemma308}The total mass within the system is conserved, that is $M(t)$ remains constant for all $t\geq0$.
\end{lemma}
\begin{proof} Since $\left(G_K(t)\right)_{t\geq0}$ is an honest semigroup, by Definition~\ref{definition286} we have
\begin{align}\label{equation316}
\frac{\ddn}{\ddn t}M_C(t)
&=\frac{\ddn}{\ddn t}\int_{N}^{\infty}\left(u_C(t)\right)(x)\,x\dd x=-c\left(u_C(t)\right)\nonumber\\
                  &=-\int_{N}^{\infty}\left(x-\int_{N}^{x}yb(y|x)\dd y\right)a(x)\left(u_C(t)\right)(x)\dd x.
\end{align}
Similarly, for the mass within the discrete regime we have
\begin{align}
&\frac{\ddn}{\ddn t}M_D(t)=\sum_{i=1}^{N}i\frac{\ddn}{\ddn t}u_{Di}(t)=\sum_{i=1}^{N}i\left((A_D+B_D)[u_D(t)]+C[u_F(t)]\right)_i\nonumber\\
&=\sum_{i=1}^{N}i\left( -a_i u_{Di}(t)+\hspace{-2mm}\sum_{j=i+1}^{N}\hspace{-2mm}a_j b_{i,j} u_{Dj}(t)+\hspace{-1mm}\int_{N}^{\infty}\hspace{-2mm}a(y)b_i(y)\left(u_{C}(t)\right)(y)\dd y\right)\nonumber\\
&=-\sum_{i=2}^{N}\left(i-\sum_{j=1}^{i-1}j b_{j,i}\right)a_i u_{Di}(t)+\int_{N}^{\infty}\left(\sum_{i=1}^{N}ib_i(y)\right)a(y)\left(u_{C}(t)\right)(y) \dd y \nonumber\\
&=\int_{N}^{\infty}\left(\sum_{i=1}^{N}ib_i(y)\right)a(y)\left(u_{C}(t)\right)(y)\dd y.\label{equation317}
\end{align}

\noindent The term $a_1 u_{D1}(t)$ is dropped from the first summation in going to the third line since $a_1=0$.
In the final step, the loss of the first term is a consequence of condition \eqref{equation304}. The rate of change of the total mass $M(t)$ is given by the addition of \eqref{equation316} and \eqref{equation317}, which by \eqref{equation303} yields
\begin{align}
\frac{\ddn}{\ddn t}M(t)=&-\int_{N}^{\infty}\left(x-\int_{N}^{x}yb(y|x)\dd y\right)a(x)\left(u_C(t)\right)(x)\dd x\nonumber\\
&+\int_{N}^{\infty}\left(\sum_{i=1}^{N}ib_i(x)\right)a(x)\left(u_{C}(t)\right)(x)\dd x\nonumber\\
=&-\int_{N}^{\infty}\left(x-\int_{N}^{x}yb(y|x)\dd y-\sum_{i=1}^{N}ib_i(x)\right)a(x)\left(u_C(t)\right)(x)\dd x\nonumber\\
=&\hspace{2mm}0,\nonumber
\end{align}
thus confirming that $M(t)$ remains constant for all $t\geq0$, and so mass is conserved.
\end{proof}

\section{Mixed Discrete--Continuous Example Model}\label{section71}

\noindent The concept of a hybrid discrete--continuous model, as developed in this paper, was proposed as a solution to  `shattering' mass-loss. We shall now introduce a particular class of such models, with the aim of confirming the findings of this work. The model is based upon the power law fragmentation model, with the continuous regime equation \eqref{equation301}, being specified by
\begin{equation}\label{equation701}
a(x)=x^\alpha,\hspace{4mm}\alpha\in\mathbb{R},\hspace{5mm}\text{and}\hspace{5mm}b(x|y)=(\nu+2)\frac{x^\nu}{y^{\nu+1}},\hspace{4mm}-2<\nu\leq0,
\end{equation}
for $N<x\leq y$. The analogous purely continuous model corresponds to the class considered by McGrady and Ziff in \cite{ziff87} and found to display mass-loss through `shattering' in the case that $\alpha<0$. With such a model selection, our continuous regime equation becomes

\begin{align}\label{equation702}
\frac{\partial u_C(x,t)}{\partial t}&=-x^\alpha u_C(x,t)+(\nu+2)x^\nu \int_{x}^{\infty} y^{\alpha-\nu-1}u_C(y,t)\dd y, \hspace{1.92mm}x>N,\hspace{1.2mm}t>0.
\end{align}

\noindent Moving on to the discrete regime equation, \eqref{equation302}, in specifying our discrete regime equation we must provide a set of continuous to discrete mass distribution functions $b_i(y)$, $i=1,\ldots,N$, such that condition~\eqref{equation303} is satisfied. If we take these functions to be given by
\begin{align*}
b_i(y)=\frac{i^{\nu+2}-(i-1)^{\nu+2}}{iy^{\nu+1}}, \hspace{5mm}y>N,\hspace{1.2mm}i=1,\ldots,N,
\end{align*}
where the value of $\nu$ is the same as in~\eqref{equation701} and~\eqref{equation702}, then it is easily verified that condition~\eqref{equation303} is satisfied.\\

\noindent Finally we must specify a choice for the discrete fragmentation parameter values $a_i$ and $b_{i,j}$, where the $b_{i,j}$ must satisfy~\eqref{equation304}. There are a number of different choices for these values considered in the literature. However, for our model we shall take the case of
uniform binary fragmentation, whereby two particles are produced from each fragmentation event, and all admissible pairings of resulting particle sizes are equally likely. This is obtained by setting
\[b_{i,j}=\frac{2}{j-1}, \hspace{5mm}i=1,\ldots,N-1,\hspace{1.2mm}j=i,\ldots,N.\]
Primarily we have selected this particular model for its simplicity, however versions of this model were studied in \cite{simha41}, one of the earliest articles on discrete fragmentation, and also in the paper \cite{ziff92}. \\

\noindent When it comes to the selection of the values $a_i$, there are minimal restrictions which must be satisfied and we can largely select any nonnegative values we wish. However, whereas $a_i=1$, was selected in \cite{simha41} and $a_i=(i-1)/(i+1)$, in \cite{ziff92}, we shall take
\[a_i=i^\alpha, \hspace{5mm}i=2,\ldots,N,\]
with $a_1=0$, in order to mirror the choice for the continuous fragmentation rate $a(x)$. Taking these selections leads to the following set of equations for the discrete regime:
\begin{align}\label{equation704}
\frac{\ddn u_{Di}(t)}{\ddn t} =&-i^\alpha u_{Di}(t)+\sum_{j=i+1}^{N} \frac{2j^\alpha}{j-1} u_{Dj}(t)\nonumber\\
&+\frac{i^{\nu+2}-(i-1)^{\nu+2}}{i}\int_{N}^{\infty}y^{\alpha-\nu-1} u_{\small{C}}(y,t)\dd y,
\end{align}
for $i=1,\ldots,N$ and $t>0$, where we lose the $-i^\alpha u_{Di}(t)$ term for $i=1$ and the summation term disappears for $i=N$.

\subsection{Exact Solutions}

\noindent The continuous regime equation~\eqref{equation702}, coincides with the one provided for the `fragmentation state' in \cite[Section 3]{huang96}. When this equation
is coupled with an initial continuous mass distribution $c_0(x)$, then \cite[Equation (9)]{huang96} gives the solution as
\begin{align}\label{equation705}
u_C(x,t)\hspace{-.5mm}= \hspace{-.5mm}{\mathrm{e}}^{-x^\alpha t}
 \hspace{-.5mm}\left\lbrace  \hspace{-.5mm}c_0(x) \hspace{-.5mm}+ \hspace{-.5mm} m\hspace{.35mm}\alpha \hspace{.35mm} t\hspace{.35mm} x^\nu \hspace{-1.5mm}\int_x^\infty \hspace{-2.5mm} y^{\alpha-\nu-1}c_0(y)_{1}F_1\left(1-m,2,t(x^\alpha\hspace{-.7mm}-y^\alpha)\right) \hspace{-.5mm} \dd y\right\rbrace,
\end{align}
\noindent where $m=(2+\nu)/\alpha$ and $_{1}F_1$ is the confluent hypergeometric function. To solve the set of discrete regime equations~\eqref{equation704}, we write them as the system
\begin{equation}\label{equation706}
\frac{\ddn}{\ddn t}\underline{u}_D(t)=E\underline{u}_D(t)+\underline{F}(t),
\end{equation}
where $\underline{u}_D(t)=(u_{D1}(t),\ldots,u_{DN}(t))^T$, the $N\times N$ matrix $E$ has the entries
\[
e_{i,j}=\left\{ \begin{aligned}
&\hspace{6mm}0\hspace{8.8mm} \textup{for} \hspace{2mm} i=j=1\hspace{1mm} \textup{or} \hspace{1mm}i>j, \\
&\hspace{2mm} -i^\alpha \hspace{6.5mm} \textup{for} \hspace{2mm} i=j>1,\\
&\hspace{2mm}\frac{2j^\alpha}{j-1}\hspace{5.2mm} \textup{for} \hspace{2mm} j>i,\\
\end{aligned} \right.
\]
\noindent and where $\underline{F}:[0,\infty)\rightarrow \mathbb{R}^N$ has the components
\begin{equation}\label{equation707}
F_i(t)=\underbrace{\frac{i^{\nu+2}-(i-1)^{\nu+2}}{i}}_{\beta_i(\nu)}\int_{N}^{\infty} y^{\alpha-\nu-1}u_C(y,t)\dd y,
\end{equation}
\noindent for $i=1,\ldots,N$, and $u_C(y,t)$ given by~\eqref{equation705}. When the initial state of the vector
$\underline{u}_D(t)$ is $d_0$, then the system~\eqref{equation706} has the solution
\begin{equation*}
\underline{u}_D(t)={\mathrm{e}}^{E t}\left\lbrace d_0+\int_0^t{\mathrm{e}}^{-sE}\left(\int_{N}^{\infty} y^{\alpha-\nu-1}u_C(y,s)\dd y\right)\hspace{-.5mm}\dd s\hspace{.7mm}\underline{\beta}(\nu)\right\rbrace,
\end{equation*}
where $\underline{\beta}(\nu)$ is the vector of values $(\beta_1(\nu),\ldots,\beta_N(\nu))^T$ from~\eqref{equation707}, and ${\mathrm{e}}^{M}$  denotes the matrix exponential of a matrix $M$.

\subsection{Example Case 1  ($\alpha=-1$ and $\nu=0$)}
\noindent Taking the model introduced above, we examine the particular case of $\alpha=-1$ and $\nu=0$. Taking such a choice of parameters in the standard continuous model has been shown to result in a shattering process \cite{ziff87}. The cut-off parameter $N$ was set at $5$ and a truncated uniform initial mass distribution imposed, whereby ${d_0}_i=1$ for $i=1,\ldots,5$ and
\begin{equation}\label{IC}
c_{0}(x)=\left\{ \begin{aligned}
\hspace{2mm} 1\hspace{4mm} &\textup{for} \hspace{2mm} 5<x<15, \\
\hspace{2mm} 0\hspace{4mm} &\textup{for} \hspace{2mm} x\geq 15.
\end{aligned} \right.
\end{equation}
The final time $T$ was taken as $100$, which gave sufficient time for this particular system to settle to its equilibrium. Below can be seen a selection of charts depicting the mass distribution at a selection of time points through the evolution of the system, starting from the uniform initial state through to the final equilibrium state.

\begin{figure}[H]
\begin{minipage}[b]{0.45\linewidth}
\centering
\includegraphics[scale=0.5]{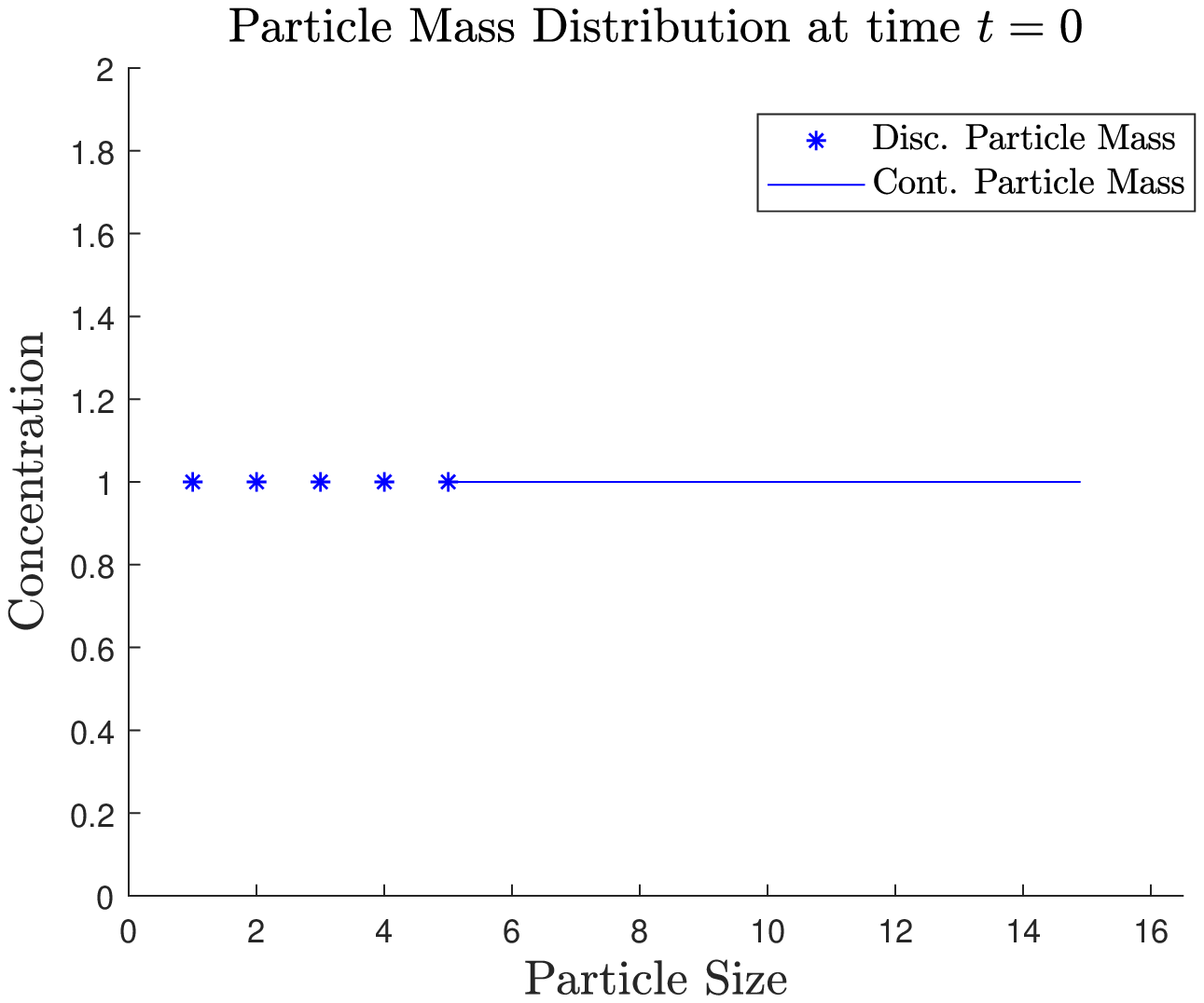}\\\hspace{0mm}\\
\includegraphics[scale=0.5]{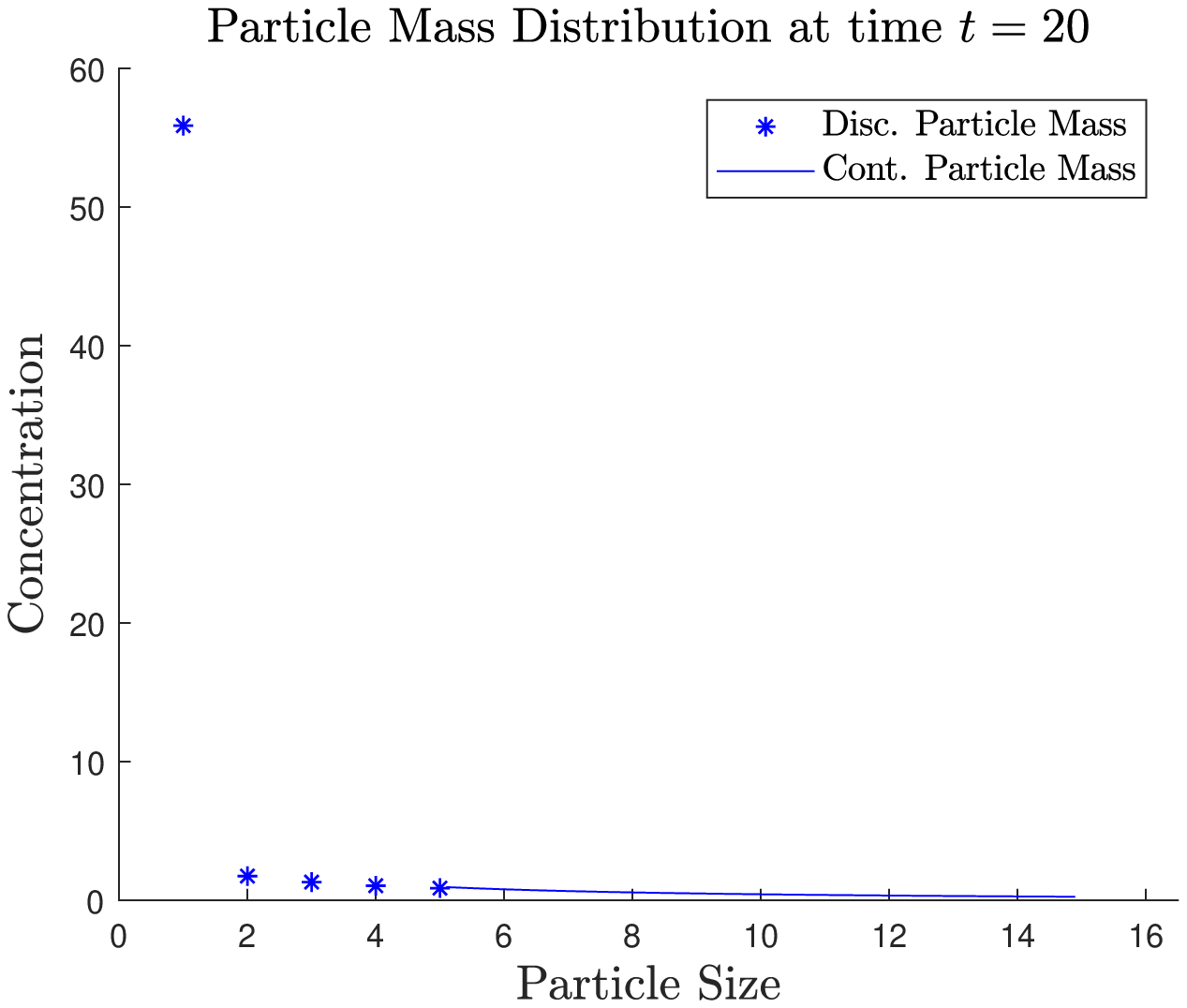}
\end{minipage}
\hspace{0.5cm}
\begin{minipage}[b]{0.45\linewidth}
\centering
\includegraphics[scale=0.5]{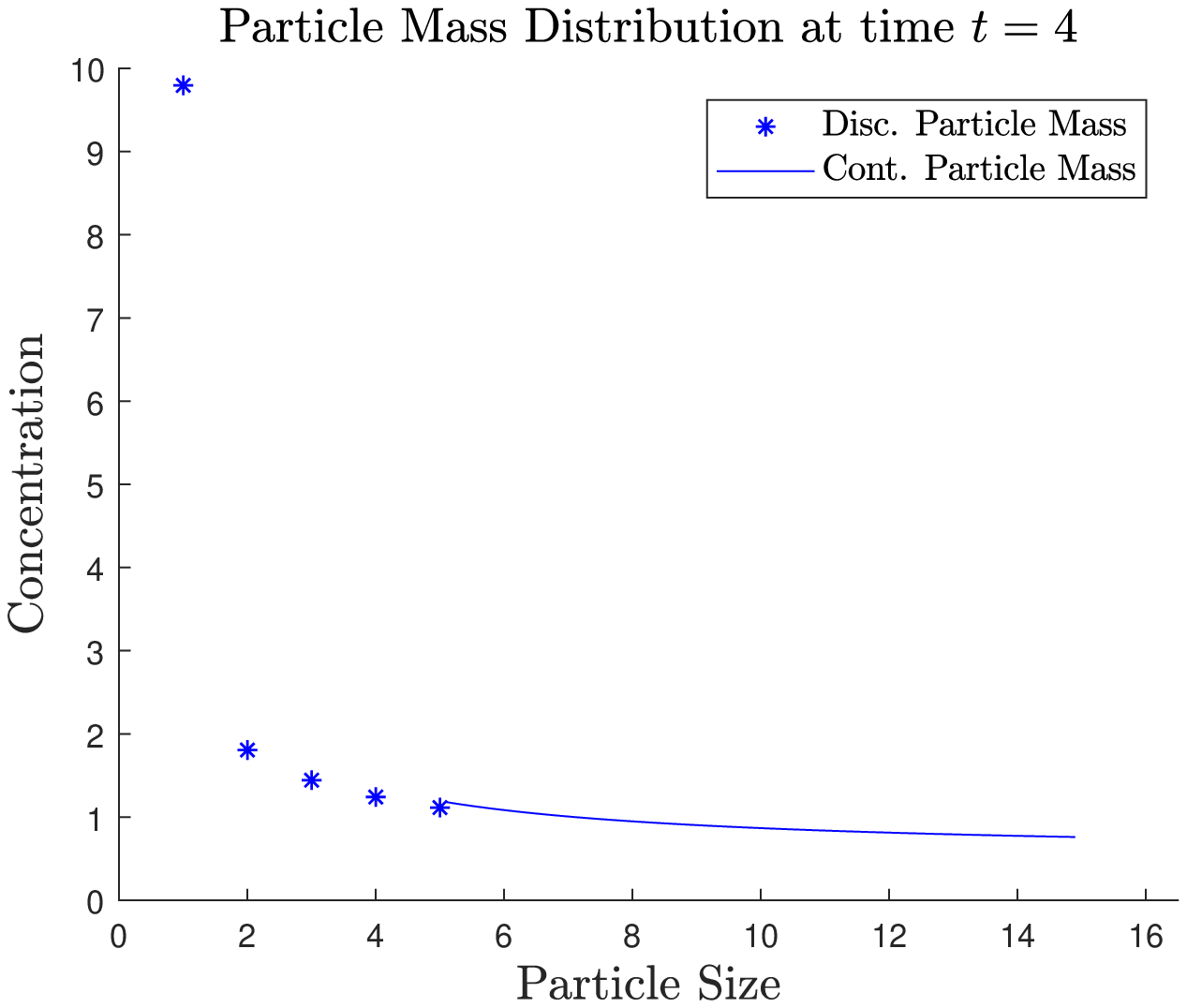}\\\hspace{0mm}\\
\includegraphics[scale=0.5]{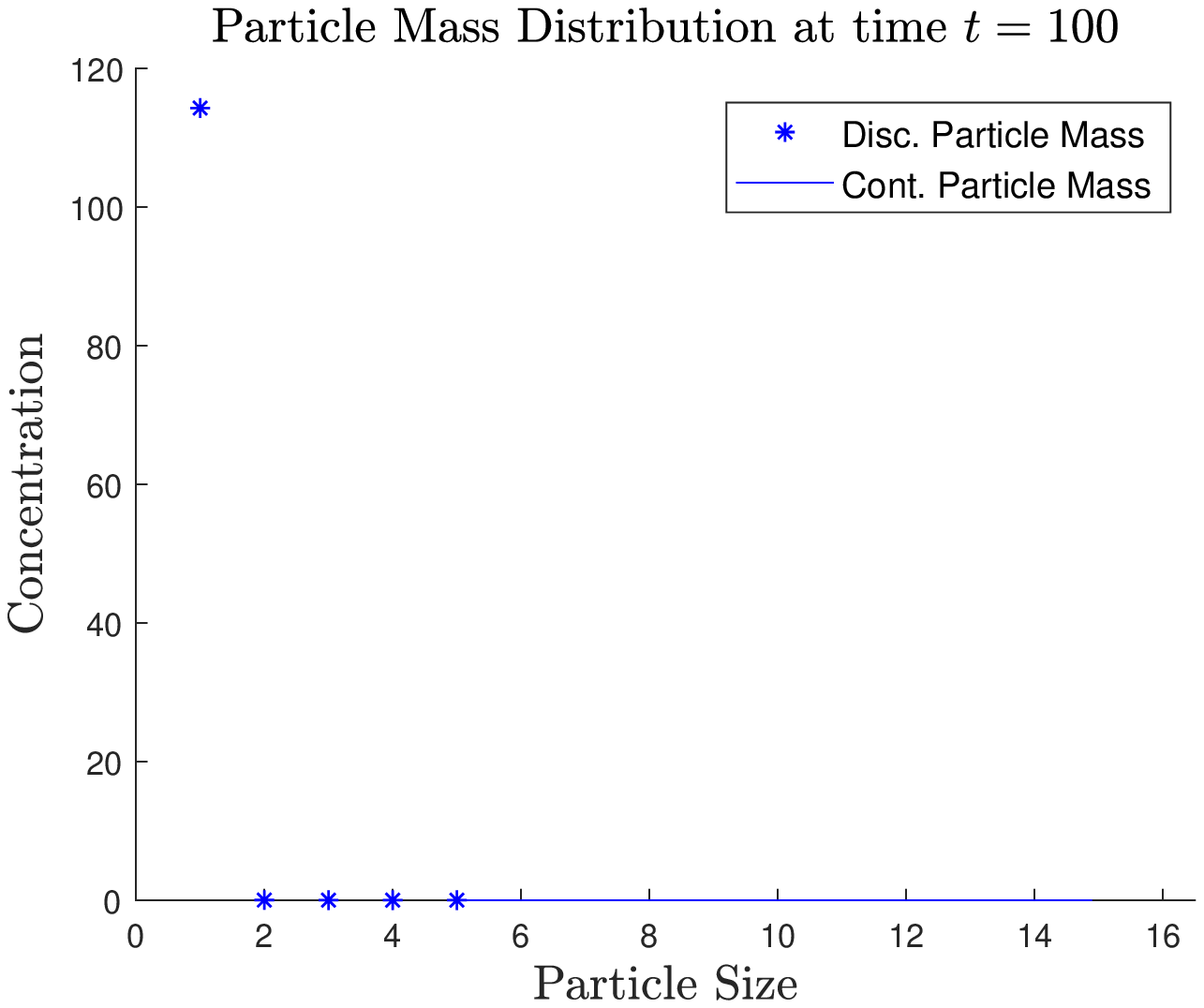}
\end{minipage}
\caption{Particle mass distribution at time $t=0,4,20$ and $100$.}\label{figure701}
\end{figure}

\noindent The above charts depict the behaviour one would expect given the physical nature of the model, with the mass becoming increasingly concentrated amongst the smaller sized particles. Additionally, the solution can be seen to remain nonnegative, as predicted by Lemma~\ref{lemma307}.
However, to get a clearer picture of whether the issue of `shattering' has been resolved we must examine the evolution of the total mass within the system.

\begin{figure}[H]
\centering
\includegraphics[scale=0.8]{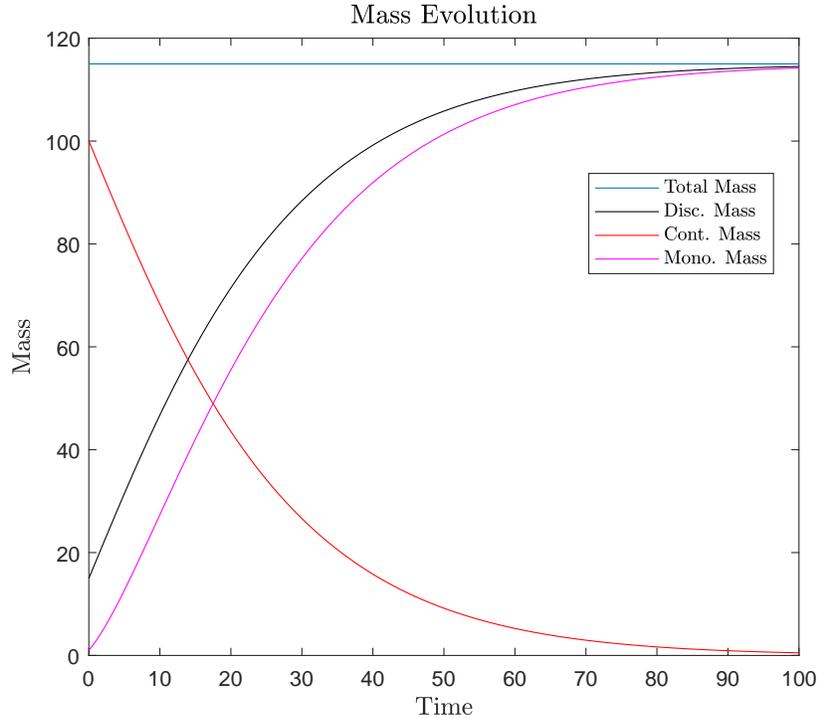}
\caption{The evolution of masses with time.}\label{figure702}
\end{figure}
\noindent Figure~\ref{figure702} details the evolution of the total mass (blue), along with the mass accounted for by the continuous regime (red), the total mass within the discrete regime (black) and the total mass accounted for by monomers (magenta).\\

\noindent As we would hope and as was predicted by Lemma~\ref{lemma308}, the total mass within the system remains constant, with a reduction in the continuous regime total mass being balanced by a gain in the total mass of the discrete regime, resulting from the fragmentation of larger continuous mass particles into smaller discrete mass particles. As time evolves, the mass accounted for by monomers grows to form an ever larger proportion of the total mass, until we reach a stage where they constitute the vast majority of the total mass and the system approaches an equilibrium.\\

\subsection{Example Case 2 ($\alpha=0.5$ and $\nu=-0.5$)}

\noindent In order to investigate the model behaviour and whether it fits with our physical intuition regarding the system, we vary the model parameters and observe the effect on the computed solutions. In this example we set $\alpha =0.5$ and $\nu=-0.5$, which has the effect of increasing the fragmentation rate and changing the
resulting size distribution for fragmentation events, to favour smaller particles. We would expect both of these changes to speed up the fragmentation process, and for equilibrium to be reached quicker than in the previous case. As before, we selected $N=5$, with the same initial state \eqref{IC}. The final time $T$ was taken as $5$, which we found to be sufficient for the system to reach a (near) equilibrium state and gave rise to the graphs in Figure~\ref{figure704}.

\begin{figure}[H]
\centering
\includegraphics[scale=0.8]{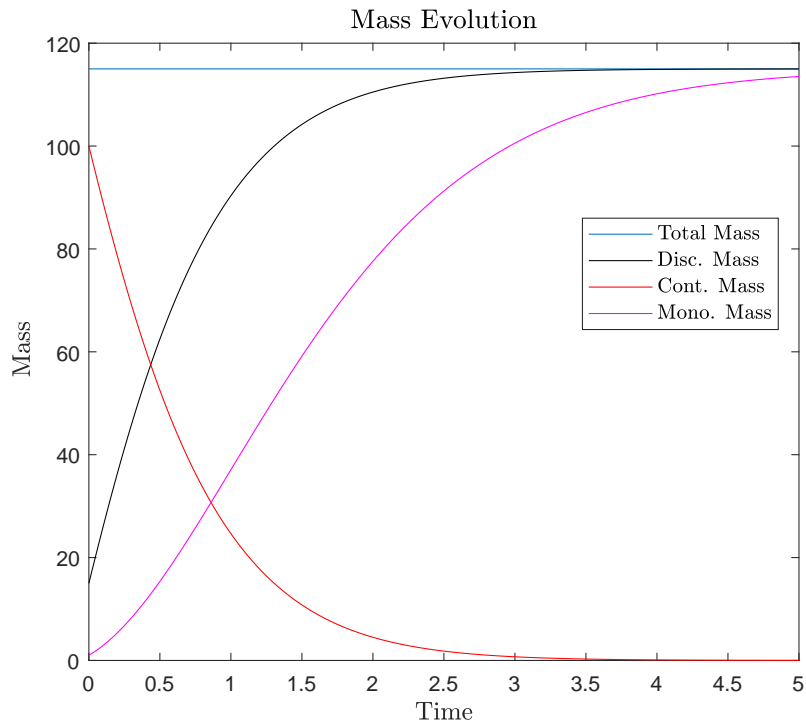}
\caption{The evolution of masses with time.}\label{figure704}
\end{figure}

\noindent As before, we see that the total mass (blue) is conserved with the loss from the continuous regime (red) being balanced by an increase in the discrete regime (black). However if we compare the model behaviour to that observed in Figure~\ref{figure702}, we see that the process reaches an equilibrium state significantly quicker than in the previous case, as expected from an intuitive consideration of the model setup.

\section{Conclusions}
\noindent In this work we introduced a mixed discrete--continuous fragmentation model \eqref{equation301} and~\eqref{equation302}. Utilising the theory of operator semigroups we were able to prove that under certain restrictions on the fragmentation rate, there exists a unique, strong solution to our system within the setting of the appropriate Banach space. In turn, this enabled us to establish the existence of a unique classical solution to our equations. Finally, we showed that these solutions preserve nonnegativity and conserve mass, two properties we would expect from a physically relevant solution. We have illustrated the theoretical results by considering a specific example of such a mixed model, an example based upon existing models, which appear commonly in the literature of the field. The examples corroborated the analysis of the paper, with the solutions displaying the nonnegativity and mass conservation predicted.

\section*{Acknowledgments}
\noindent This work was supported by the UK Engineering and Physical Sciences Research Council [EP/J500495/1  03].

%\section*{References}
\bibliographystyle{elsarticle-num}
\bibliography{refs}

\end{document}